\documentclass{amsart}

\usepackage{hyperref}
\usepackage{cite}

\newtheorem{theorem}{Theorem}[section]
\newtheorem{lemma}[theorem]{Lemma}

\theoremstyle{definition}
\newtheorem{definition}[theorem]{Definition}

\theoremstyle{remark}
\newtheorem{remark}[theorem]{Remark}

\numberwithin{equation}{section}

\newcommand{\Addresses}{{
  \bigskip
  \footnotesize

  Xavier ~Ramos Oliv\'e, \textsc{Department of Mathematics, University of California, Riverside, Riverside, CA 92521, USA} 
  
  \par\nopagebreak
  \textit{E-mail address}: \texttt{olive@math.ucr.edu}

}}

\newcommand{\tphi}{\widetilde{\varphi}}
\newcommand{\wtg}{\widetilde{g}}
\newcommand{\heat}{\left( \Delta - \partial_t \right)}
\newcommand{\hop}{\mathcal{L}}

\begin{document}

\title[Neumann Li-Yau gradient estimate]{Neumann Li-Yau gradient estimate under integral Ricci curvature bounds}

\author{Xavier Ramos Oliv\'e}






\begin{abstract}
We prove a Li-Yau gradient estimate for positive solutions to the heat equation, with Neumann boundary conditions, on a compact Riemannian submanifold with boundary ${\bf M}^n\subseteq {\bf N}^n$, satisfying the integral Ricci curvature assumption:
\begin{equation}
D^2 \sup_{x\in {\bf N}} \left( \oint_{B(x,D)} |Ric^-|^p dy \right)^{\frac{1}{p}} < K 
\end{equation}
for $K(n,p)$ small enough, $p>n/2$, and $diam({\bf M})\leq D$. The boundary of ${\bf M}$ is not necessarily convex, but it needs to satisfy the interior rolling $R-$ball condition. 
\end{abstract}

\maketitle

\section{Introduction}

Let $({\bf M}^n,g)$ be a Riemannian manifold with boundary $\partial {\bf M}$. In \cite{LiYau}, P. Li and S.T. Yau proved a series of Li-Yau gradient estimates for positive solutions to the heat equation on $\bf M$. In particular, they proved that if $\bf M$ is a compact manifold with $Ric \geq -K$, for some $K\geq 0$, and the boundary of $\bf M$ is convex (i.e. its second fundamental form is nonnegative $II\geq 0$), then any positive solution $u(x,t)$ to the Neumann problem:

\begin{align} \label{NeumannProblem}
 \begin{cases}
  \partial_t u -\Delta u =0 & \text{ in } \overset{\circ}{{\bf M}} \times (0,\infty)\\
  \partial_\nu u =0 & \text{ on }\partial {\bf M} \times (0,\infty) \\
 \end{cases}
\end{align}

where $\overset{\circ}{{\bf M}} = {\bf M}\setminus \partial {\bf M}$ and $\nu$ denotes the outer unit normal vector to $\partial {\bf M}$, satisfies:

\begin{equation} \label{LiYau}
\frac{|\nabla u|^2}{u^2} - \alpha \frac{\partial_t u}{u}\leq C_1 +C_2\frac{1}{t} 
\end{equation}

for all $\alpha >1$, where:

\begin{align*}
 C_1 &= \frac{n}{\sqrt{2}}\alpha^2(\alpha -1)^{-1}K\\
 C_2 &= \frac{n}{2}\alpha^2
\end{align*}

Later, in \cite{Wang}, J. Wang generalized this result to a case where $\partial {\bf M}$ is not necessarily convex; more precisely, he considers the case where $II\geq -H$ for some $H\geq 0$, adding the necessary ``interior rolling $R-$ball'' condition, inspired by the work of R. Chen in \cite{Chen}.

\begin{definition}
 A Riemannian manifold with boundary ${\bf M}$ is said to satisfy the \emph{interior rolling $R-$ball condition} if for any $p\in \partial {\bf M}$ there exists $q\in {\bf M}$ such that $B(q,R)\subseteq {\bf M}$, and $B(q,R)\cap \partial {\bf M} = \{ p\}$, where $B(q,R)$ denotes the geodesic ball centered at $q$ with radius $R$.
\end{definition}

Under these assumptions, J. Wang proved that $u(x,t)$ satisfies the Li-Yau gradient estimate (\ref{LiYau}), with constants:

\begin{align*}
 C_1 &= \frac{6n\alpha(\alpha -1)(1+H)^7K}{(\alpha -(1+H)^2)^2}+\frac{309n^2\alpha^3(\alpha -1)(1+H)^{10}H}{(\alpha - (1+H)^2)^4R^2\beta}\\
 C_2 &= \frac{n\alpha^2(\alpha-1)^2(1+H)^4}{(2-\beta)(1-\beta)(\alpha - (1+H)^2)^2}
\end{align*}

for any $\alpha >(1+H)^2$ and any $0<\beta<\frac{1}{2}$.\\

 Notice that in both of the examples above, the authors were assuming a pointwise lower bound on the Ricci curvature. This pointwise assumption has been recently weakened by Q.S. Zhang and M. Zhu, in \cite{ZhangZhu1} and \cite{ZhangZhu2}. One of their results is a Li-Yau type gradient estimate under the integral Ricci curvature bounds introduced by P. Petersen and G. Wei in \cite{PetersenWei1} and \cite{PetersenWei2}.
\vskip 1em

In particular, in \cite{ZhangZhu2} they prove the Li-Yau type gradient estimate:

\begin{equation}
\alpha \underline{J} \frac{|\nabla u|^2}{u^2} - \frac{\partial_t u}{u} \leq \frac{C_1}{\underline{J}} \left[1+\frac{C_2}{\underline{J}}\right]+ \frac{C_3}{\underline{J}} \frac{1}{t}  
\end{equation}

for the heat kernel on a manifold ${\bf N}$, where $C_1$, $C_2$ and $C_3$ are constants depending on $n$, $p$ and $\alpha$, and $\underline{J} = \underline{J}(t)$ is a decreasing exponential function (see \cite{ZhangZhu2} for more details). Their curvature assumption is that, for $\kappa$ small enough, $p>n/2$:

\begin{equation}
\sup_{x\in N} \left( \oint_{B(x,1)} |Ric^{-}|^p dV \right)^{1/p} \leq \kappa 
\end{equation}

where $|Ric^-|$ is the negative part of the Ricci curvature, i.e. $|Ric^-| =  \max \{0, -\rho(x)\}$ where $\rho(x)$ denotes the lowest eigenvalue of the Ricci curvature. Here  $\oint$ denotes the average integral over the domain. The smallness of the integral Ricci curvature is a necessary condition in general, as shown by X. Dai, G. Wei, and Z. Zhang in \cite{DWZ}. Note that this result is for manifolds without boundary. 

%
%
%
%
\vskip 1em

To deal with a manifold with boundary with integral Ricci curvature bounds, we use the technique developed in \cite{Wang} together with the technique developed in \cite{ZhangZhu1} and \cite{ZhangZhu2}. We consider a manifold ${\bf N}^n$ satisfying the scaling invariant curvature condition:
\begin{equation}
D^2 \sup_{x\in {\bf N}} \left( \oint_{B(x,D)} |Ric^-|^p dy \right)^{\frac{1}{p}} < K 
\end{equation}
with $p>n/2$ and $K(n,p)>0$ small enough to have the volume doubling property proved in \cite{PetersenWei2}. Let ${\bf M}^n\subseteq {\bf N}^n$ be a compact Riemannian submanifold with boundary, $diam({\bf M})\leq D$, whose boundary is not necessarily convex, but that satisfies the interior rolling $R-$ball condition. Then we derive a Li-Yau gradient estimate for the positive solutions $u(x,t)>0$ to the problem (\ref{NeumannProblem}).

\vskip 1em
More precisely, our main theorem is:

\begin{theorem}
Given $H>0$, $n>0$, $p>\frac{n}{2}$, and $R>0$ small enough, there exists $K(n,p)>0$ such that if ${\bf M}^n$ is a compact Riemannian submanifold with boundary of a Riemannian manifold ${\bf N}^n$ with the properties:

\begin{enumerate}
 \item $D^2 \sup_{x\in {\bf N}} \left( \oint_{B(x,D)} |Ric^-|^p dy \right)^{\frac{1}{p}} < K $, where $diam({\bf M})\leq D$
\vskip 1em
 \item $II \geq - H$, where $II$ is the second fundamental form of $\partial{\bf M}$
\vskip 1em
 \item ${\bf M}$ satisfies the interior rolling $R-$ball condition
\end{enumerate}
\vskip 1em
then any positive solution $u(x,t)$ to: 
 \begin{equation}
  \begin{cases}
   \partial_t u-\Delta u = 0 & \text{in }\overset{\circ}{{\bf M}} \times (0,\infty)\\
   \partial_{\nu} u = 0 & \text{on }\partial {\bf M}\times (0,\infty)
  \end{cases}
 \end{equation}
 satisfies the Li-Yau type gradient estimate:
 \begin{equation}
 \alpha \underline{J} \frac{|\nabla u|^2}{u^2} - \frac{\partial_t u}{u} \leq C_1 + \frac{C_2}{\underline{J}}\frac{1}{t}
 \end{equation}
where, given any $0<\xi<1$, we can choose any $0<\alpha \leq \frac{1-\xi}{(1+H)^2}$ and any $0<\beta \leq \frac{\xi^2(1-\xi)}{2\xi^2 + n^2(1+H)^2}$, and where:

\begin{equation}
\begin{split}
  C_1 = \frac{n^2}{\alpha \sqrt{2\xi^3(1-2\beta)}} \Bigg( \frac{32n^2\alpha H^2(1+H)^2}{\xi^3 R^2} + 2\alpha (1+H)\left[ \frac{H}{R^2} + \frac{2(n-1)H(3H+1)}{R}\right] +\\
 + (\beta +4\alpha^{-1})\left[ \frac{4\alpha H(1+H)}{R} \right]^2 \Bigg)
\end{split}
\end{equation}
\begin{equation}
 C_2 = \frac{n^2}{\alpha (1-2\beta)}
\end{equation}
\begin{equation}
 \underline{J}(t) :=  2^{-\frac{1}{c-1}}e^{-\frac{\widetilde{C_3}}{c-1}t} 
\end{equation}
where  $c = (3+\frac{1}{\alpha})\frac{1}{\beta}$ and 
\begin{equation}
 \widetilde{C_3} = C_3(\alpha , \beta, n, p)\left[ \frac{K}{D^{2-\frac{n}{p}}R^{\frac{n}{p}}} + \frac{K^{\frac{2p}{2p-n}}}{D^{\frac{4p-6n}{2p-n}}R^{\frac{4n}{2p-n}}}\right]>0
\end{equation}

\end{theorem}

\begin{remark}
 Notice that the range of values of $\alpha$ for which this theorem holds is consistent with the range of values of $\alpha$ in \cite{Wang}.
\end{remark}

One of the key tools that we used are the Gaussian upper bounds of the Neumann heat kernel proved by M. Choulli, L. Kayser, and E.M. Ouhabaz in  \cite{CKO}. To apply their result, we use the volume doubling results of \cite{PetersenWei2}, and we show in Lemma \ref{VDonM} that the interior rolling $R-$ball condition ensures that the volume doubling property in $\bf M$ holds up to the boundary. These two results, together with the Sobolev inequality of \cite{DWZ}, allow us to use the Gaussian upper bounds to find $\underline{J}(t)$, as in \cite{ZhangZhu2}.


\section{Proof of main theorem}

As in the proof of \cite{Wang} and \cite{Chen}, consider a nonnegative $C^2$ function $\psi(r)$ defined on $[0,\infty)$ such that

$$\begin{cases}
 \psi(r) \leq H & \text{if }r\in [0,1/2)\\
 \psi(r) = H & \text{if }r\in [1,\infty)
\end{cases}
$$

with $\psi(0) = 0$, $0\leq \psi'(r) \leq 2H$, $\psi'(0) = H$ and $\psi''(r) \geq -H$. Define $\phi(x) := \psi \left( \frac{r(x)}{R} \right)$, where $r(x)$ denotes the distance from $x\in {\bf M}$ to $\partial {\bf M}$. Let $\varphi(x) := (1+\phi(x))^2$, and $\tphi(x):= \alpha \varphi(x)$, where $\alpha>0$ will be determined below.

\begin{lemma}\label{lemma1}
 The function $\tphi$ satisfies:
 \begin{align}
   &\alpha \leq \tphi \leq \alpha (1+H)^2 \label{lemma1.1}\\ 
   &|\nabla \tphi| \leq \frac{4\alpha H(1+H)}{R} \label{lemma1.2}\\
   &\Delta \tphi \geq -2\alpha (1+H)\left[ \frac{H}{R^2}+\frac{2(n-1)H(3H+1)}{R}\right] \label{lemma1.3}
 \end{align}

\end{lemma}
\begin{proof}[Proof of Lemma 2.1]
The inequalities (\ref{lemma1.1}) and (\ref{lemma1.2}) are immediate from the definitions. To prove (\ref{lemma1.3}) we follow the same argument as in \cite{Wang} and \cite{Chen}: we need to use that 

\begin{equation}\label{keyineq}
\Delta r \geq -(n-1)(3H+1) 
\end{equation}

For a detailed argument on how to derive (\ref{keyineq}) see \cite{Chen}. Then:

\begin{equation}
 \Delta \phi = \frac{\psi '' |\nabla r|^2}{R^2} + \frac{\psi ' \Delta r}{R}\geq -\frac{H}{R^2} - \frac{2H(n-1)(3H+1)}{R}
\end{equation}

So:

\begin{equation}
 \Delta \tphi = 2\alpha |\nabla \phi|^2 + 2\alpha (1+\phi)\Delta \phi \geq -2\alpha(1+H)\left[ \frac{H}{R^2} + \frac{2H(n-1)(3H+1)}{R} \right]
\end{equation}
 
\end{proof}

\begin{remark}
 The inequality (\ref{keyineq}) holds as long as $R<1$ is chosen small enough so that:
 
 \begin{equation}
  \sqrt{K_R} \tan (R\sqrt{K_R}) \leq \frac{1+H}{2}
  \end{equation}
and
  \begin{equation}
  \frac{H}{\sqrt{K_R}} \tan (R\sqrt{K_R}) \leq \frac{1}{2}
  \end{equation}
where $K_R$ is the supremum of the sectional curvature at distance $R$ from the boundary.
\end{remark}


Before we start the proof of the main theorem, we will need the following technical lemma which will be proved in the following section.

\begin{lemma}\label{lemma2}
 There exists a unique smooth solution $J(x,t)$ to the problem:
 
\begin{equation}
\begin{cases}
\Delta J -\partial_t J -c \frac{|\nabla J|^2}{J} - 2J|Ric^-|=0 & \text{ in }\overset{\circ}{{\bf M}}\times(0,\infty)\\
\partial_\nu J = 0 & \text{ on } \partial{\bf M}\times (0,\infty)\\
J = 1 & \text{ on } {\bf M}\times \{0\}
\end{cases} 
\end{equation}

for $c>1$ constant, and it satisfies 

\begin{equation} \label{Jbounds}
0< \underline{J} \leq J \leq 1 
\end{equation}

where $\underline{J} = \underline{J} (t)$ is given by 

\begin{equation}
 \underline{J}(t) :=  2^{-\frac{1}{c-1}}e^{-\frac{\widetilde{C_3}}{c-1}t}  
\end{equation}

and $\widetilde{C_3} = C_3(c, n,p)\left[ \frac{K}{D^{2-\frac{n}{p}}R^{\frac{n}{p}}} + \frac{K^{\frac{2p}{2p-n}}}{D^{\frac{4p-6n}{2p-n}}R^{\frac{4n}{2p-n}}}\right]>0$.
 
 \begin{remark}
  The function $\underline{J} (t)$ is decreasing in $t$.
 \end{remark}
\end{lemma}

The proof of this lemma is provided in section \ref{prooflemma2}. With it, we can start the proof of our main theorem.

\begin{proof}[Proof of Theorem 1.2]
Consider the function:

\begin{equation}
G(x,t):= t\left[\tphi J (|\nabla f|^2 +\epsilon)-\partial_t f \right] 
\end{equation}

where $f=\ln u$, $\epsilon>0$, and $J(x,t)$ is the function from Lemma \ref{lemma2} corresponding to $c=\left( 3+\frac{1}{\alpha}\right)\frac{1}{\beta}$, where $\alpha, \beta>0$ are constants. 
\vskip 1em
Let $(p,t_0)$ be the maximum of $G$ in ${\bf M}\times [0,T]$, for $T>0$. Notice that we can assume w.l.o.g. that $t_0>0$, since otherwise $t_0=0$ implies that $G\leq 0$ in ${\bf M}\times [0,T]$, which is stronger than what we want to prove.

\vskip 1em
\emph{Case 1: $p\in \partial {\bf M}$}
\vskip 1em

In that case, $\partial_\nu G(p,t_0) \geq 0$, and choosing an orthonormal frame at $p$ so that $e_n=\nu$, we get:

\begin{equation}
 0 \leq t_0 \left[ \partial_\nu \tphi J \left(|\nabla f|^2+\epsilon \right)+ \tphi \left( \partial_\nu J \left( |\nabla f|^2 + \epsilon \right)+2J\sum_{i=1}^{n} \partial_i f \partial_\nu \partial_i  f \right) -\partial_\nu \partial_t f \right]
\end{equation}

Using that $\partial_\nu f = 0$ and $\partial_\nu J = 0$ on $\partial {\bf M}\times (0,\infty)$, and dividing by $t_0 \tphi J(|\nabla f|^2+\epsilon)$, we get that at $(p,t_0)$:

\begin{equation}
 0\leq \frac{1}{\tphi} \partial_\nu \tphi + 2\frac{\sum_{i=1}^{n-1} \partial_i f \partial_\nu \partial_i f}{|\nabla f|^2+\epsilon }
\end{equation}

But now, following the argument of \cite{Wang}, this leads to a contradiction. By direct computation:

\begin{equation}
 \sum_{i=1}^{n-1} \partial_i f \partial_\nu \partial_i f = -II(\nabla f, \nabla f) \leq H |\nabla f|^2
\end{equation}

Also, since $p\in \partial {\bf M}$ we have that $r(p)=0$, and so $\phi(p) = \psi(0)= 0$; so using that $\psi'(0) = H$ and that $\nabla r \cdot \nu (p) = -1$ we get:

\begin{equation}
 \partial_\nu \tphi (p) = \nabla \tphi \cdot \nu (p) = 2\alpha (1+\phi(p))\psi'\left(\frac{r(p)}{R}\right) \frac{1}{R}\nabla r\cdot \nu = -\alpha \frac{2H}{R}
\end{equation}

Now, from the two expressions above, assuming w.l.o.g. that $R<1$ and using that $\tphi (p) = \alpha$, for any $\epsilon>0$ we get:

\begin{equation}
\frac{1}{\tphi}\partial_\nu\tphi + 2\frac{\sum_{i=1}^{n-1} \partial_i f \partial_\nu \partial_i f}{|\nabla f|^2+\epsilon }\leq -\frac{2H}{R} + 2\frac{H |\nabla f|^2}{|\nabla f|^2+\epsilon } <0
\end{equation}

which is a contradiction. Thus, \emph{Case 1} can not occur.

\vskip 1em
\emph{Case 2: $p\in \overset{\circ}{{\bf M}} = {\bf M}\setminus \partial {\bf M}$}
\vskip 1em

In this case, $(p,t_0)$ is a local maximum, thus $\nabla G(p,t_0) = 0$, $\partial_t G(p,t_0) \geq 0$, and $\Delta G(p,t_0) \leq 0$, which implies that $\Delta G - \partial_t G \leq 0$ at $(p,t_0)$. We can assume w.l.o.g. that $G(p,t_0) >0$, since otherwise we get a stronger statement than what we are going to prove.
\vskip 1em
Now we proceed similarly as in \cite{ZhangZhu2}. By direct computation, using Bochner's formula:

\begin{equation}
 \heat \left( \frac{|\nabla u|^2}{u} + \epsilon u \right) = \frac{2}{u} \left|\partial_i \partial_j u - \frac{\partial_i u\partial_j u}{u}\right|^2 + 2R_{ij}\frac{\partial_i u\partial_j u}{u}
\end{equation}

Let's call $g= \frac{|\nabla u|^2}{u^2} +\epsilon$ and $\wtg = ug = \frac{|\nabla u|^2}{u} + \epsilon u$. Then:

\begin{align}\label{eq1}
 \heat \left( \tphi J\wtg - \partial_tu \right)  &= \heat(\tphi J)\cdot \wtg + 2\nabla (\tphi J)\nabla \wtg + \tphi J \left[ \heat\wtg \right] \\ &= \heat(\tphi J)\cdot \wtg  + 2\nabla(\tphi J)\nabla \wtg + \tphi J \left[ \frac{2}{u} \left|\partial_i \partial_j u - \frac{\partial_i u\partial_j u}{u}\right|^2 + 2R_{ij}\frac{\partial_i u\partial_j u}{u} \right]
\end{align}

Using the quotient formula for the operator $\hop = \Delta - \partial_t$, which is:

\begin{equation}
 \hop \left( \frac{A}{B} \right) + 2\nabla \ln B \nabla \frac{A}{B} = \frac{\hop A}{B} - \frac{A\hop B}{B^2}
\end{equation}

for $A= \tphi J \wtg -\partial_t u$ and $B=u$, and defining $Q=G/t = \tphi Jg-\frac{\partial_t u}{u}$, we get:

\begin{equation}
 \heat  Q  + \frac{2}{u}\nabla u \nabla Q = \frac{\heat (\tphi J \wtg - \partial_t u)}{u}
\end{equation}

And from (\ref{eq1}) we get:

\begin{equation}\label{eq2}
 \heat  Q  + \frac{2}{u}\nabla u \nabla Q = \heat(\tphi J)\cdot g  + 2\nabla(\tphi J)\frac{\nabla \wtg}{u} + \tphi J \left[ \frac{2}{u^2} \left|\partial_i \partial_j u - \frac{\partial_i u\partial_j u}{u}\right|^2 + 2R_{ij}\frac{\partial_i u\partial_j u}{u^2} \right]
\end{equation}

Also, using that $\nabla \left( \frac{|\nabla u|^2}{u} \right)\frac{1}{u} = \nabla \left( \frac{|\nabla u|^2}{u^2}\right) + \frac{|\nabla u|^2}{u^2}\frac{\nabla u}{u}$, and the notation $f=\ln u$, we observe that:

\begin{equation}
 \frac{\nabla \wtg}{u} = \nabla \left( \frac{|\nabla u|^2}{u^2}\right) + \frac{|\nabla u|^2}{u^2}\frac{\nabla u}{u} +\epsilon\frac{\nabla u}{u} =  \nabla \left(|\nabla f|^2\right) +g\nabla f
\end{equation}

Hence (\ref{eq2}) becomes:

\begin{equation}\label{eq3}
 \heat  Q  + 2\nabla f \nabla Q = \heat(\tphi J)\cdot g  + 2\nabla(\tphi J)\left[ \nabla \left( |\nabla f|^2\right) +g\nabla f\right] + 2\tphi J \left[ \left|\partial_i \partial_j f\right|^2 + Ric(\nabla f, \nabla f) \right]
\end{equation}

Now notice that, for $\beta>0$:

\begin{equation}
 \nabla J \nabla(|\nabla f|^2) \geq  -\frac{1}{\beta J}|\nabla J|^2|\nabla f|^2 - \beta J |\partial_i \partial_j f|^2 
\end{equation}

So (\ref{eq3}) gives us:

\begin{equation}
\begin{split}
 \heat  Q  + 2\nabla f \nabla Q \geq \heat(\tphi J)\cdot g  + 2g\nabla(\tphi J)\nabla f + 2J\nabla\tphi \nabla (|\nabla f|^2) \\
 -\frac{2}{\beta J}\tphi|\nabla J|^2|\nabla f|^2 - 2\beta \tphi J |\partial_i \partial_j f|^2  + 2\tphi J \left[ \left|\partial_i \partial_j f\right|^2 + Ric(\nabla f, \nabla f) \right] 
\end{split}
\end{equation}

Since $Ric(\nabla f, \nabla f) \geq -|Ric^-||\nabla f|^2$ and $ |\nabla f|^2\leq g$:

\begin{equation}
 \begin{split}
 \heat  Q  + 2\nabla f \nabla Q \geq \heat(\tphi J)\cdot g  + 2g\nabla(\tphi J)\nabla f + 2J\nabla\tphi \nabla (|\nabla f|^2) \\ 
 -\frac{2}{\beta J}\tphi g|\nabla J|^2 - 2\beta \tphi J |\partial_i \partial_j f|^2  + 2\tphi J \left[ \left|\partial_i \partial_j f\right|^2 -g|Ric^-|) \right]= \\
 = \heat(\tphi J)\cdot g  + 2J\nabla\tphi \nabla (|\nabla f|^2) +2\tphi J \left[ (1-\beta)|\partial_i \partial_j f|^2 -g|Ric^-| \right]\\ 
 -\frac{2}{\beta J}\tphi g|\nabla J|^2 + 2g\tphi \nabla J \nabla f + 2gJ\nabla \tphi \nabla f
 \end{split}
\end{equation}

Using that $-2|\nabla f||\nabla J| \geq - \frac{|\nabla J|^2}{\beta J} - \beta J |\nabla f|^2$ and Cauchy-Schwarz, we get:

\begin{equation}
 \begin{split}
  \heat  Q  + 2\nabla f \nabla Q  \geq \heat(\tphi J)\cdot g  + 2(1-\beta)\tphi J |\partial_i \partial_j f|^2 - \beta J g \tphi |\nabla f|^2 \\ 
  + 2J\nabla \tphi \nabla(|\nabla f|^2) + 2gJ\nabla \tphi \nabla f + \left[ -\frac{3}{\beta} \frac{|\nabla J|^2}{J} - 2J|Ric^-|\right]\tphi g
 \end{split}
\end{equation}

Now expanding the first term on the right:

\begin{equation}
 \begin{split}
  \heat  Q  + 2\nabla f \nabla Q  \geq g\left[ \Delta \tphi J + 2\nabla \tphi \nabla J\right]  + 2(1-\beta)\tphi J |\partial_i \partial_j f|^2 - \beta J g \tphi |\nabla f|^2 \\ 
  + 2J\nabla \tphi \nabla(|\nabla f|^2) + 2gJ\nabla \tphi \nabla f + \left[\Delta J -\partial_t J -\frac{3}{\beta} \frac{|\nabla J|^2}{J} - 2J|Ric^-|\right]\tphi g
 \end{split}
\end{equation}

Using (\ref{lemma1.1}) we get $2\nabla \tphi \nabla J \geq -\frac{1}{\alpha \beta}\tphi  \frac{|\nabla J|^2}{J} -\beta J |\nabla \tphi|^2$, hence:

\begin{equation}\label{eq4}
 \begin{split}
  \heat  Q  + 2\nabla f \nabla Q  \geq Jg\Delta \tphi + 2J\nabla \tphi \nabla(|\nabla f|^2) + 2gJ\nabla \tphi \nabla f - \beta |\nabla \tphi|^2gJ\\ 
  +  2(1-\beta)\tphi J |\partial_i \partial_j f|^2 - \beta J g \tphi |\nabla f|^2 + \tphi g \left[\Delta J -\partial_t J -\left(3 + \frac{1}{\alpha}\right)\frac{1}{\beta} \frac{|\nabla J|^2}{J} - 2J|Ric^-|\right]
 \end{split}
\end{equation}

Since $J$ solves the problem of Lemma \ref{lemma2} for $c = \left(3 +\frac{1}{\alpha}\right)\frac{1}{\beta}$, we see that (\ref{eq4}) becomes:

\begin{equation}
 \begin{split}
  \heat  Q  + 2\nabla f \nabla Q  \geq J \Bigg[ g\Delta \tphi + 2\nabla \tphi \nabla(|\nabla f|^2) + 2g\nabla \tphi \nabla f - \beta |\nabla \tphi|^2g\\ 
  +  2(1-\beta)\tphi  |\partial_i \partial_j f|^2 - \beta  g \tphi |\nabla f|^2 \Bigg] 
 \end{split}
\end{equation}

Note that:

\begin{equation}
 \heat G + 2\nabla f\nabla G = \heat (tQ)+2\nabla f\nabla (tQ) = t\left[ \heat Q +2\nabla f \nabla Q \right] - Q
\end{equation}

Since we know that $(p,t_0)$ is a local maximum, $\heat G(p,t_0) + 2\nabla f\nabla G (p,t_0) = \heat G(p,t_0) \leq 0$, so at $(p,t_0)$ we have:

\begin{equation}
 0\geq t_0 J\left[ g\Delta \tphi + 2\nabla \tphi \nabla(|\nabla f|^2) + 2g\nabla \tphi \nabla f - \beta |\nabla \tphi|^2g
  +  2(1-\beta)\tphi  |\partial_i \partial_j f|^2 - \beta  g \tphi |\nabla f|^2  \right] - Q
\end{equation}

Expanding the second term on the right, notice that:

\begin{equation}
 2\nabla \tphi \nabla (|\nabla f|^2) 
 \geq -4\alpha^{-1}|\nabla\tphi|^2|\nabla f|^2-\alpha |\partial_i \partial_j f|^2
\end{equation}

 Hence:

\begin{equation}
 0\geq t_0 J\left[ g\Delta \tphi + 2g\nabla \tphi \nabla f - \beta |\nabla \tphi|^2g
  +  (2(1-\beta)\tphi-\alpha)  |\partial_i \partial_j f|^2 -4\alpha^{-1}|\nabla \tphi|^2|\nabla f|^2- \beta  g \tphi |\nabla f|^2  \right] - Q
\end{equation}

Now using that

\begin{equation}
\sum_{i,j=1}^n|\partial_i\partial_j f|^2 \geq \frac{1}{n^2} \left( \sum_{i,j=1}^n |\partial_i \partial_j f| \right)^2 \geq \frac{1}{n^2} (\Delta f)^2 = \frac{1}{n^2}(|\nabla f|^2- \partial_t f)^2 
\end{equation}

we get:

\begin{equation}
 0\geq t_0 J\left[ g\Delta \tphi + 2g\nabla \tphi \nabla f - \beta |\nabla \tphi|^2g
  +  \frac{2(1-\beta)\tphi-\alpha}{n^2}  (|\nabla f|^2-\partial_t f)^2 -4\alpha^{-1}|\nabla \tphi|^2|\nabla f|^2- \beta  g \tphi |\nabla f|^2  \right] - Q
\end{equation}

In the discussion below, $O(\epsilon)$ denotes a function that goes to zero as $\epsilon$ goes to zero. Expanding the terms containing $g= |\nabla f|^2 + \epsilon$, by Lemma \ref{lemma1} and the elementary inequality $2\nabla \tphi \nabla f \geq -(|\nabla \tphi|^2 + |\nabla f|^2)$, we get:

\begin{equation}
\begin{split}
0\geq t_0 J\Bigg[ |\nabla f|^2\Delta \tphi + 2|\nabla f|^2\nabla \tphi \nabla f -\epsilon |\nabla f|^2- \beta |\nabla \tphi|^2|\nabla f|^2  \\ 
+  \frac{2(1-\beta)\tphi-\alpha}{n^2}  (|\nabla f|^2-\partial_t f)^2 -4\alpha^{-1}|\nabla \tphi|^2|\nabla f|^2- \beta  \epsilon \tphi |\nabla f|^2 -\beta \tphi |\nabla f|^4 \Bigg] - Q +O(\epsilon)
\end{split}
\end{equation}

Using Cauchy-Schwarz and rearranging terms:

\begin{equation}\label{eq5}
\begin{split}
 0\geq t_0 J \Bigg[ (\Delta \tphi -(\beta + 4\alpha^{-1}) |\nabla \tphi|^2 +O(\epsilon ))|\nabla f|^2 -2|\nabla \tphi||\nabla f|^3 - \beta \tphi |\nabla f|^4  \\+ \frac{2(1-\beta)\tphi-\alpha}{n^2}  (|\nabla f|^2-\partial_t f)^2 \Bigg] -Q +O(\epsilon)
\end{split}
\end{equation}

Now we will relate the term $(|\nabla f|^2-\partial_t f)^2$ to $Q^2$. Notice that, by direct computation:

\begin{equation}
Q^2 = (|\nabla f|^2-\partial_t f)^2 + (\tphi^2 J^2 -1)|\nabla f|^4 + 2(1-\tphi J)|\nabla f|^2 \partial_t f + O(\epsilon)|\nabla f|^2 +O(\epsilon) \partial_t f +O(\epsilon) 
\end{equation}

Using that $\partial_t f = \tphi J (|\nabla f|^2 + \epsilon) - Q$, we get:

\begin{equation}
Q^2 = (|\nabla f|^2-\partial_t f)^2 -(1-\tphi J)^2 |\nabla f|^4 - 2(1-\tphi J)Q|\nabla f|^2 + O(\epsilon)|\nabla f|^2 + O(\epsilon) Q+O(\epsilon) 
\end{equation}

Hence:

\begin{equation}
(|\nabla f|^2-\partial_t f)^2 = Q^2 +(1-\tphi J)^2 |\nabla f|^4 + 2(1-\tphi J)Q|\nabla f|^2 + O(\epsilon)|\nabla f|^2 + O(\epsilon) Q+O(\epsilon)
\end{equation}

Notice that if we choose $\alpha \leq \frac{1-\xi}{(1+H)^2}$ for some $0<\xi <1$, from (\ref{lemma1.1}) and (\ref{Jbounds}) we get $1-\tphi J\geq \xi$, so $2(1-\tphi J)Q|\nabla f|^2\geq 0$, hence:

\begin{equation}
(|\nabla f|^2-\partial_t f)^2 \geq Q^2 +(1-\tphi J)^2 |\nabla f|^4 + O(\epsilon)|\nabla f|^2 + O(\epsilon) Q+O(\epsilon)
\end{equation}

Making sure that our later choice of $\beta$ is so that $\frac{2(1-\beta)\tphi - \alpha}{n^2}>0$, (\ref{eq5}) becomes:

\begin{equation}\label{eq6}
\begin{split}
0 \geq t_0 J \Bigg[ \left( \Delta \tphi - (\beta +4\alpha^{-1})|\nabla \tphi|^2 +O(\epsilon ) \right)|\nabla f|^2 - 2|\nabla \tphi||\nabla f|^3 +\\ +\left(\left[\frac{2(1-\beta)\tphi-\alpha}{n^2} \right](1-\tphi J)^2 - \beta \tphi \right) |\nabla f|^4 + \frac{2(1-\beta)\tphi-\alpha}{n^2}Q^2 \Bigg] -\left(1+O(\epsilon)\right)Q + O(\epsilon)
\end{split}
\end{equation}

Now we are going to choose $\beta >0$ so that the coefficient of $|\nabla f|^4$ is positive; namely, for some $A >0$, we want to have:

\begin{equation}
\left[\frac{2(1-\beta)\tphi-\alpha}{n^2} \right](1-\tphi J)^2 - \beta \tphi \geq A
\end{equation}

Using Lemmas \ref{lemma1} and \ref{lemma2}, and the choice of $\alpha$, we know that:

\begin{equation}
\left[\frac{2(1-\beta)\tphi-\alpha}{n^2} \right](1-\tphi J)^2 - \beta \tphi \geq  \left[\frac{2(1-\beta)\alpha-\alpha}{n^2} \right]\xi^2 - \beta \alpha (1+H)^2
\end{equation}

Now setting the right hand side to be greater or equal than $A$, we get the condition:
 
 \begin{equation}
  \beta \leq \frac{\alpha \xi^2 - A n^2}{2\alpha \xi^2 + \alpha n^2 (1+H)^2}
 \end{equation}

 To ensure that there are positive values of $\beta$, we choose $A = \frac{\alpha \xi^3}{n^2}$, so that the condition above becomes:
 
 \begin{equation}
  \beta \leq \frac{\xi^2(1-\xi)}{2\xi^2 + n^2(1+H)^2}
 \end{equation}
 
 Choosing $\beta$ in this way and using Lemmas \ref{lemma1} and \ref{lemma2}, the coefficient of $|\nabla f|^2$ satisfies:
 
 \begin{equation}
  \Delta \tphi -(\beta + 4\alpha^{-1})|\nabla \tphi|^2 + O(\epsilon) \geq -C(\alpha , \beta , n,H,R)+O(\epsilon) =:-C_\epsilon
 \end{equation}
 
 where:
 
 \begin{equation}
  C := 2\alpha (1+H)\left[ \frac{H}{R^2} + \frac{2(n-1)H(3H+1)}{R}\right] + (\beta +4\alpha^{-1})\left[ \frac{4\alpha H(1+H)}{R} \right]^2 
 \end{equation}

 The one of $|\nabla f|^3$ satisfies:
 
 \begin{equation}
  -2|\nabla \tphi| \geq -\frac{8\alpha H(1+H)}{R} =: -B(\alpha, H, R)
 \end{equation}
 
 Finally, the one of $Q^2$ satisfies:
 
 \begin{equation}
  \frac{2(1-\beta)\tphi-\alpha}{n^2} \geq \frac{\alpha (1-2\beta)}{n^2} =: E(\alpha, \beta, n)
 \end{equation}

 So (\ref{eq6}) becomes:
 
 \begin{equation}\label{eq7}
0 \geq t_0 J \left[ -C_\epsilon|\nabla f|^2-B|\nabla f|^3 + A |\nabla f|^4 + EQ^2 \right] -\left(1+O(\epsilon ) \right)Q + O(\epsilon)
 \end{equation}

 Now following the same argument as in \cite{Wang}, calling $y= |\nabla f|^2$, notice that:
 
\begin{align}
 &Ay^2-By^{3/2}-C_\epsilon y = \\
 =&\frac{A}{2}y^2 +  \left( \sqrt{\frac{A}{2}}y-\frac{B}{\sqrt{2A}}y^{1/2} \right)^2 - \frac{B^2}{2A}y - C_\epsilon y \geq\\
 \geq & \left( \sqrt{\frac{A}{2}}y-\frac{1}{\sqrt{2A}} \left(\frac{B^2}{2A} +C_\epsilon \right) \right)^2 - \frac{1}{2A} \left( \frac{B^2}{2A} +C_\epsilon \right)^2 \geq \\
 \geq & - \frac{1}{2A} \left( \frac{B^2}{2A} +C_\epsilon \right)^2 = - \frac{1}{2A} \left( \frac{B^2}{2A} +C \right)^2 + O(\epsilon)=: -\widetilde{D}+O(\epsilon)
\end{align}

Thus (\ref{eq7}) becomes:

\begin{equation}
0 \geq Et_0J Q^2 -(1+O(\epsilon)) Q -\widetilde{D}t_0J +O(\epsilon) 
\end{equation}

Or equivalently, multiplying by $t_0$:

\begin{equation}
 0 \geq EJ G^2 - (1+O(\epsilon)) G - \widetilde{D}t_0^2 J + O(\epsilon)
\end{equation}

Notice that the right hand side is quadratic in $G$, with the leading coefficient being positive. So if it is nonpositive, we must have:

\begin{equation}
 G(p,t_0) \leq \frac{1}{2EJ(t_0)}+\sqrt{\frac{1}{4E^2J^2(t_0)}+\frac{\widetilde{D}t_0^2}{E} +O(\epsilon)} +O(\epsilon)
\end{equation}

Since $(p,t_0)$ is the maximum of $G$ in ${\bf M}\times [0,T]$ and using Lemma \ref{lemma2} and that $\underline{J}$ is decreasing in $t$, we get that for any $x\in {\bf M}$:

\begin{equation}
 G(x,T) \leq \frac{1}{2E\underline{J}(T)}+\sqrt{\frac{1}{4E^2\underline{J}^2(T)}+\frac{\widetilde{D}T^2}{E} +O(\epsilon)} +O(\epsilon)
\end{equation}

So using again Lemma \ref{lemma2}:

\begin{equation}
 T\left[ \tphi \underline{J}(|\nabla f|^2+\epsilon)-\partial_t f \right] \leq G \leq \frac{1}{2E\underline{J}}+\sqrt{\frac{1}{4E^2\underline{J}^2}+\frac{\widetilde{D}T^2}{E} +O(\epsilon)} +O(\epsilon)
\end{equation}

where all the functions in the previous inequality are being evaluated at $(x,T)$. At this point, the inequality does not depend on the point $(p,t_0)$ (which could change when we vary $\epsilon$), and since the inequality holds for any $\epsilon >0$, we get:

\begin{equation}
 T\left[ \tphi \underline{J}|\nabla f|^2-\partial_t f \right] \leq \frac{1}{2E\underline{J}}+\sqrt{\frac{1}{4E^2\underline{J}^2}+\frac{\widetilde{D}T^2}{E} } 
\end{equation}

The argument above works for any value of $T>0$, so at any point $(x,t)\in {\bf M}\times(0,\infty)$ we have:

\begin{equation}
 t\left[ \tphi \underline{J}|\nabla f|^2-\partial_t f \right] \leq \frac{1}{2E\underline{J}}+\sqrt{\frac{1}{4E^2\underline{J}^2}+\frac{\widetilde{D}t^2}{E} } \leq \frac{1}{E\underline{J}}+t\sqrt{\frac{\widetilde{D}}{E}} 
\end{equation}

Hence, using Lemma \ref{lemma1}, $f = \ln (u)$, and dividing by $t$, we get what we wanted:

\begin{equation}
 \alpha \underline{J}\frac{|\nabla u|^2}{u^2} - \frac{\partial_t u}{u} \leq C_1 + \frac{C_2}{\underline{J}}\frac{1}{t} 
\end{equation}

where $C_1 = \sqrt{\frac{\widetilde{D}}{E}}$ and $C_2 = \frac{1}{E}$.

\end{proof}

\section{Proof of Lemma \ref{lemma2}}\label{prooflemma2}

Using the transformation $w = J^{-(c-1)}$, the problem from Lemma \ref{lemma2} becomes:

\begin{equation}
\begin{cases}
\Delta w -\partial_t w + Vw = 0  & \text{ in }\overset{\circ}{{\bf M}}\times(0,\infty)\\
\partial_\nu w = 0 & \text{ on } \partial{\bf M}\times (0,\infty)\\
w = 1 & \text{ on } {\bf M}\times \{0\}
\end{cases} 
\end{equation}

where $V(x) := 2(c-1)|Ric^-|(x) \geq 0$. Note that this is a linear parabolic PDE with Neumann boundary conditions, which has a unique smooth solution given by Duhamel's formula:

\begin{align}
 w(x,t) &= \int_{\bf M} 1\cdot h(t,x,y)dy + \int_0^t \int_{\bf M} h(t-s,x,y) V(y)w(y,s)dyds =\\
 &= 1 + \int_0^t \int_{\bf M} h(t-s,x,y) V(y)w(y,s)dyds
\end{align}

where $h(t,x,y)\geq 0$ is the Neumann heat kernel on $\bf M$.

\vskip 1em
{\bf Claim 1:} $w(x,t) > 0$
\vskip 1em

\begin{proof}
 Argue by contradiction. First, define:
 
 \begin{equation}
  \underline{w}(t) = \min_{x\in {\bf M}} w(x,t)
 \end{equation}
 
 If the claim is false, there exists some $t\geq 0$ such that $\underline{w}(t) \leq 0$. Since $\underline{w}$ is continuous and $\underline{w}(0) = 1$, there exists $t_0>0$ such that $\underline{w}(t_0)=0$ and $\underline{w}(t)>0$ for any $0\leq t<t_0$. Let $x_0\in {\bf M}$ be a point that realizes the minimum $w(x_0,t_0) = \underline{w}(t_0) = 0$. Then:

 \begin{equation}
  \int_0^{t_0} \int_{\bf M} h(t_0-s,x_0,y) V(y)w(y,s)dyds \geq 0
 \end{equation}

 So:
 
\begin{equation}
  w(x_0,t_0) \geq  1
\end{equation}
 
 But that's a contradiction. 
 
\end{proof}

By the previous claim, we know that $J = w^{-\frac{1}{c-1}}$ is well defined and smooth. Moreover:

\vskip 1em
{\bf Claim 2:} $J(x,t)\leq 1$
\vskip 1em

\begin{proof}
 This statement is equivalent to $w(x,t) \geq 1$. By a similar argument as above, since $w>0$:

  \begin{equation}
  \int_0^{t} \int_{\bf M} h(t-s,x,y) V(y)w(y,s)dyds \geq 0
 \end{equation}

 hence:
 
 \begin{equation}
  w(x,t) \geq  1
 \end{equation}
 
\end{proof}

\vskip 1em
{\bf Claim 3:} $J(x,t) \geq \underline{J}(t)$ where $ \underline{J}(t) :=  2^{-\frac{1}{c-1}}e^{-\frac{\widetilde{C_3}}{c-1}t} $ and $\widetilde{C_3} = C_3(c, n,p)\left[ \frac{K}{D^{2-\frac{n}{p}}R^{\frac{n}{p}}} + \frac{K^{\frac{2p}{2p-n}}}{D^{\frac{4p-6n}{2p-n}}R^{\frac{4n}{2p-n}}}\right]$.
\vskip 1em

\begin{proof}
 To prove this we will follow a similar argument as in \cite{ZhangZhu2}: we will find Gaussian upper bounds for $w(x,t)$. To do so, we use the following result from Choulli, Kayser, and Ouhabaz in \cite{CKO} (Theorem 1.1): 
 
 \begin{lemma}[Choulli, Kayser, Ouhabaz \cite{CKO}]\label{lemma3}
  Suppose that ${\bf N}^n$ is a smooth Riemannian manifold that satisfies the volume doubling property and whose heat kernel $p(t,x,y)$ satisfies:
  
  \begin{equation}\label{heatk}
   p(t,x,y)\leq \frac{C}{[V(x,\sqrt{t})V(y,\sqrt{t})]^{1/2}}e^{-c\frac{d(x,y)^2}{t}}
  \end{equation}

  where $C,c>0$ are constants, $d(x,y)$ denotes the geodesic distance between $x,y\in {\bf N}^n$, and $V(x,r) = Vol (B(x,r))$ is the volume of a geodesic ball of radius $r>0$. Suppose also that ${\bf M}^n\subseteq {\bf N}^n$ is a Riemannian submanifold with boundary such that $diam ({\bf M}) < \infty$, satisfying the volume doubling property: 
  
  \begin{equation}\label{VDM}
   V_{\bf M}(x,s) \leq \widetilde{C} \left(\frac{s}{r}\right)^\gamma V_{\bf M}(x,r)
  \end{equation}

  where $V_{\bf M} (x,r)= Vol (B(x,r)\cap {\bf M})$ is the volume in ${\bf M}^n$ of a geodesic ball $B(x,r)$ of ${\bf N}^n$, and $\widetilde{C},\gamma>0$ are positive constants. Then, the Neumann heat kernel on ${\bf M}^n$ satisfies:
  
  \begin{equation}
   h(t,x,y) \leq \frac{C}{[V_{\bf M}(x,\sqrt{t})V_{\bf M}(y,\sqrt{t})]^{1/2}} \left( 1 + \frac{d(x,y)^2}{4t} \right)^\gamma e^{-\frac{d(x,y)^2}{4t}}
  \end{equation}
 \end{lemma}

 \vskip 2em
 All the hypothesis of this lemma are known to be satisfied, except maybe for property (\ref{VDM}) near the boundary. More precisely, $diam({\bf M})<\infty$ since $\bf M$ is compact, and the volume doubling property for $\bf N$ was proved in \cite{PetersenWei2} and can be stated as:
 
 \begin{lemma}[Petersen, Wei \cite{PetersenWei2}]
Given $p>n/2$ and $D >0$, there exists $K= K (n,p)>0$ such that if

\begin{equation}\label{curvcon}
D^2 \sup_{x\in {\bf N}} \left( \oint_{B(x,D)} |Ric^-|^p \right)^{1/p}< K
\end{equation}

then for all $x\in {\bf N}$ and $r<s< D$ we have:

\begin{equation}
 V(x,s) \leq 2\left( \frac{s}{r} \right)^n V(x,r)
\end{equation}

\end{lemma}

Choosing $D$ so that $diam({\bf M})\leq D$, we get the volume doubling property for all the balls completely contained in $\bf M$. Also, (\ref{heatk}) follows from the volume doubling property and a Sobolev inequality proved in \cite{DWZ}, as explained in \cite{ZhangZhu2}. Now we will show that property (\ref{VDM}) holds as well:
 
 \begin{lemma}\label{VDonM}
    Given $D,\gamma >0$, if a manifold $\bf N$ satisfies the volume doubling property
    \begin{equation}
     V(x,s) \leq C\left(\frac{s}{r}\right)^\gamma V(x,r)
    \end{equation}
    for $r\leq s\leq D$ and $x\in {\bf N}$, then a compact submanifold with boundary ${\bf M}\subseteq {\bf N}$, with $diam({\bf M}) \leq D$, whose boundary satisfies the interior rolling $R-$ball condition, satisfies:
    \begin{equation}
     V_{\bf M}(x,s) \leq \widetilde{C} \left(\frac{s}{r}\right)^\gamma V_{\bf M}(x,r)
    \end{equation}
    for $0<r\leq s $, $x\in {\bf M}$, $\widetilde{C} = \max \{ 3^\gamma C,  \left( \frac{2D}{R} \right)^\gamma C\}$.
    \end{lemma}
 \begin{proof}
  We only need to consider the situation where $B(x,s)\cap \partial {\bf M} \not = \emptyset$. If that's the case, then let $p\in \partial {\bf M}$ denote the closest point to $x$ in $\partial {\bf M}$. Using the interior rolling $R-$ball condition, we know that there exists $q\in {\bf M}$ such that $B(q,R)\subseteq {\bf M}$ and $B(q,R)\cap \partial {\bf M} = \{ p\}$.
 \vskip 1em
  Notice that, by definition, $p$ minimizes the distance to the boundary from $x$ and from $q$, hence the geodesics joining $p$ and $x$, $\gamma_{px}$, and $p$ and $q$, $\gamma_{pq}$, must both be perpendicular to $\partial {\bf M}$ at $p$. Hence, since geodesics do not branch, $x$, $q$ and $p$ are on the same geodesic.  
\vskip 1em
  \emph{Case 1:} $0<r \leq R$
\vskip 1em
  Consider a point $q'$ on the geodesic joining $x$ and $q$, such that $d(q',x) = r/2$ and $d(q',q)=d(x,q)-r/2$. Notice that $p$ also minimizes the distance from $q'$ to $\partial {\bf M}$. It's easy to see that $B(q',r/2) \subseteq B(x,r)\cap{\bf M}$, and also $B(x,s)\subseteq B(q',3s/2)$. 
  \vskip 1em
  If $\frac{3s}{2} \leq D$, using volume doubling in $\bf N$:
  
  \begin{equation}
   V_{\bf M}(x,s) \leq V(q',3s/2) \leq 3^\gamma C \left( \frac{s}{r} \right)^\gamma V(q',r/2) \leq  \widetilde{C} \left( \frac{s}{r} \right)^\gamma V_{\bf M}(x,r)
  \end{equation}

 If $\frac{3s}{2}\geq D$ we get:
  \begin{equation}
V_{\bf M}(x,s)  \leq V(q',D) \leq C \left(\frac{2D}{r}\right)^\gamma V(q',r/2) \leq 3^\gamma C\left( \frac{s}{r} \right)^\gamma V(q',r/2)\leq \widetilde{C} \left( \frac{s}{r} \right)^\gamma V_{\bf M}(x,r)    
  \end{equation}

\vskip 1em
  \emph{Case 2:} $R \leq r$
\vskip 1em
  
  In this case, consider also a point $q''$ on the geodesic joining $x$ and $q$, such that $d(q'',x) = R/2$ and $d(q'',q) = d(x,q)-R/2$. As before, it's easy to see that $B(q'',R/2)\subseteq B(x,r)\cap {\bf M}$. Thus, we can see that:
  
    \begin{equation}
      V_{\bf M}(x,s)\leq V(q'',D) \leq C \left( \frac{2D}{R}\right)^\gamma V(q'',R/2) \leq \widetilde{C} V_{\bf M} (x,r)
    \end{equation}
    
    which is stronger than what we want to prove, so in particular:
    
    \begin{equation}
      V_{\bf M}(x,s)\leq \widetilde{C} \left( \frac{s}{r}\right)^\gamma V_{\bf M} (x,r)
    \end{equation}

 \end{proof}
 
 In our case, in particular, we can get:

     \begin{equation}
      V_{\bf M}(x,s)\leq \left(2^{n+1}\cdot3^n \frac{D^n}{R^n}\right) \left( \frac{s}{r}\right)^n V_{\bf M} (x,r)
    \end{equation}
 
 Notice that, as a consequence of the volume doubling in ${\bf N}$ and the curvature condition (\ref{curvcon}), we can derive:

\begin{equation}
 \left( \oint_{\bf M} |Ric^-|^p \right)^{\frac{1}{p}} < 2^{\frac{1}{p}}  \frac{1}{D^{\frac{2p-n}{p}}R^{\frac{n}{p}}} K
\end{equation}
 
 Now, using Lemma \ref{lemma3}, and following the same argument as in \cite{ZhangZhu2}, we can finish the proof of the claim. Let 
 
 \begin{equation}
 \overline{w}(t) = \sup_{(x,s)\in {\bf M}\times[0,t]}w(x,s) 
 \end{equation}
 Then:
 \begin{equation}\label{eq8}
  w(x,t) \leq 1+ \int_0^t \overline{w}(s) \int_{\bf M} h(t-s,x,y)V(y) dyds
 \end{equation}

\vskip 1em
  \emph{Case 1:} $t-s \geq D^2$
\vskip 1em
  
In this case $V_{\bf M} (z,\sqrt{t-s}) = |M|$, so:

\begin{equation}
\begin{split}
 \int_{\bf M} h(t-s,x,y)V(y) dy \leq \int_{\bf M} \frac{Ce^{-\frac{d(x,y)^2}{4(t-s)}}}{[V_{\bf M}(x,\sqrt{t-s}) V_{\bf M}(y,\sqrt{t-s})]^{1/2}} \left(1+\frac{d(x,y)^2}{4(t-s)} \right)^n V(y) dy \leq \\
 \leq \frac{C(c, n)}{|{\bf M}|} \int_{\bf M} |Ric^-|(y) dy \leq C(c, n) \left( \oint_{\bf M} |Ric^-|^p \right)^{\frac{1}{p}} \leq C_4\frac{K}{D^{2-\frac{n}{p}}R^{\frac{n}{p}}} =: \widetilde{C_4}
\end{split}
\end{equation}
where $C_4= C_4(c , n, p)$.
 
\vskip 1em
  \emph{Case 2:} $t-s \leq D^2$
\vskip 1em 
 
 Using Lemma \ref{VDonM}:
 
 \begin{equation}
  V_{\bf M}(z,\sqrt{t-s}) \geq  \frac{R^n}{2^{n+1}\cdot 3^n D^{2n}}(t-s)^{n/2} V_{\bf M}(z,D) = C(n)\frac{R^n}{D^{2n}}(t-s)^{n/2} |{\bf M}|
 \end{equation}

Thus, by H\"older's inequality:

\begin{equation}
 \begin{split}
 \int_{\bf M} h(t-s,x,y)V(y) dy \leq || V||_{L^p({\bf M})} \left( \int_{\bf M} hh^{\frac{1}{p-1}} \right)^{\frac{p-1}{p}} \leq \\ \leq || V||_{L^p({\bf M})} \frac{C(n,p)D^{\frac{2n}{p}}}{|{\bf M}|^{\frac{1}{p}} R^{\frac{n}{p}}(t-s)^{\frac{n}{2p}}} \left( \int_{\bf M} h dy\right)^{\frac{p-1}{p}} \leq \frac{C_5K}{D^{2-\frac{3n}{p}}R^{\frac{2n}{p}}} (t-s)^{-\frac{n}{2p}}=: \widetilde{C_5}(t-s)^{-\frac{n}{2p}}
 \end{split}
\end{equation}
where $C_5 = C_5(c , n, p) $, and where we have used that:

\begin{equation}
    h(t,x,y) \leq  \frac{C}{[V_{\bf M}(x,\sqrt{t})V_{\bf M}(y,\sqrt{t})]^{1/2}}
\end{equation}

 With these estimates, (\ref{eq8}) becomes:

 \begin{equation}\label{eq9}
  w(x,t)  \leq 1 + \widetilde{C_4}\int_0^{t-D^2} \overline{w}(s)ds + \widetilde{C_5}\int_{t-D^2}^t (t-s)^{-\frac{n}{2p}}\overline{w}(s)ds
 \end{equation}

 The second term on the right can be written as:
 
 \begin{equation}
 \begin{split}
  \int_{t-D^2}^t (t-s)^{-\frac{n}{2p}}\overline{w}(s)ds = \int_{t-D^2}^{t-\epsilon} (t-s)^{-\frac{n}{2p}}\overline{w}(s)ds+\int_{t-\epsilon}^t (t-s)^{-\frac{n}{2p}}\overline{w}(s)ds\leq \\ 
  \leq \epsilon^{-\frac{n}{2p}}\int_{t-D^2}^{t-\epsilon} \overline{w}(s)ds + \overline{w}(t) \int_{t-\epsilon}^t (t-s)^{-\frac{n}{2p}} ds \leq \epsilon^{-\frac{n}{2p}}\int_{t-D^2}^{t} \overline{w}(s)ds + \frac{2p \epsilon^{\frac{2p-n}{2p}}}{2p-n}\overline{w}(t)  
 \end{split}
 \end{equation}
where we have used that  $p> \frac{n}{2}$. Then, taking supremum over $(x,\tilde{t})\in {\bf M}\times [0,t]$, (\ref{eq9}) becomes:
\begin{equation}
\left[ 1 - \frac{2p\widetilde{C_5} \epsilon^{\frac{2p-n}{2p}}}{2p-n} \right] \overline{w}(t) \leq 1+ \max \{\widetilde{C_4}, \widetilde{C_5}\epsilon^{-\frac{n}{2p}}\} \int_0^t \overline{w}(s) ds
\end{equation}
Now, choosing $\epsilon = \left(\frac{4p\widetilde{C_5}}{2p-n}\right)^{-\frac{2p}{2p-n}}$, we get:
\begin{equation}
 \overline{w}(t) \leq 2+ \widetilde{C_3}\int_0^t \overline{w}(s) ds
\end{equation}
where $\widetilde{C_3}$ is chosen so that $\widetilde{C_3}\geq 2\max \{\widetilde{C_4}, \widetilde{C_5}\epsilon^{\frac{-n}{2p}}\}$, for example:

\begin{equation}
 \widetilde{C_3} = C_3\left[ \frac{K}{D^{2-\frac{n}{p}}R^{\frac{n}{p}}} + \frac{K^{\frac{2p}{2p-n}}}{D^{\frac{4p-6n}{2p-n}}R^{\frac{4n}{2p-n}}}\right]
\end{equation}
for some constant $C_3= C_3(c, n,p)$. So by Gr\"onwall's inequality:
\begin{equation}
 w(x,t)\leq \overline{w}(t) \leq 2e^{\widetilde{C_3} t} 
\end{equation}
Thus, using $J = w^{-\frac{1}{c-1}}$, we get 

\begin{equation}
 \underline{J}(t) :=  2^{-\frac{1}{c-1}}e^{-\frac{\widetilde{C_3}}{c-1}t}  \leq J(x,t) 
\end{equation}

\end{proof}
\vskip 5em

\section*{Acknowledgements}

The author would like to thank Professors Qi S. Zhang, Meng Zhu, Lihan Wang and Guofang Wei for their valuable advise and comments.

\newpage
\bibliographystyle{alphanum}

\Addresses
\end{document}